\documentclass[12pt]{amsart}

\usepackage{tikz-cd}

\usepackage[english]{babel}
\usepackage[utf8x]{inputenc}
\usepackage[T1]{fontenc}
\usepackage{amsmath}
\usepackage{amstext}
\usepackage{amsfonts}
\usepackage{amssymb}
\usepackage{amsthm}

\usepackage[alphabetic, initials]{amsrefs}
\usepackage{mathtools}
\usepackage{dsfont}
\usepackage{color}
\usepackage{graphicx}
\usepackage[colorinlistoftodos]{todonotes}
\usepackage[colorlinks, linkcolor=red, citecolor=blue, urlcolor=blue, hypertexnames=true]{hyperref}

\usepackage{float}
\usepackage[colorinlistoftodos]{todonotes}
\usepackage{theoremref}
\usepackage{enumitem}
\usepackage{todonotes}

\allowdisplaybreaks
\newtheorem{theorem}{Theorem}[section]

\newtheorem{prop}[theorem]{Proposition}
\newtheorem{cor}[theorem]{Corollary}
\newtheorem{thm}[theorem]{Theorem}

\newtheorem{question}[theorem]{Question}
\newtheorem*{cor*}{Corollary}
\newtheorem*{conjecture*}{Conjecture}
\newtheorem*{thm*}{Theorem}
\newtheorem*{lem*}{Lemma}

\newtheorem*{prop*}{Proposition}
\theoremstyle{definition}

\newtheorem*{defn*}{Definition}

\theoremstyle{remark}

\title[On invariant subalgebras when the ISR property fails]{On invariant  subalgebras when the ISR property fails}

\author{Yongle Jiang*}

\address{Yongle Jiang, School of Mathematical Sciences, Dalian University of Technology,
Dalian, 116024, China}
\email{yonglejiang@dlut.edu.cn}

\author{Ruoyu Liu}
\address{Ruoyu Liu, School of Mathematical Sciences, Dalian University of Technology,
Dalian, 116024, China}
\address{Current Address: Dalian Jinpu New District Gaochengshan Middle School, Dalian, 
116100, China}

\email{liuruoyu0925@163.com}

\thanks{$\ast$-Corresponding author}

\date{\today}

\begin{document}

\begin{abstract}
We classify all $G$-invariant von Neumann subalgebras in $L(G)$ for $G=\mathbb{Z}^2\rtimes SL_2(\mathbb{Z})$. This is the first result on classifying $G$-invariant von Neumann subalgebras in $L(G)$ for i.c.c. groups $G$ without the invariant von Neumann subalgebras rigidity property (ISR property for short) as introduced in Amrutam-Jiang's work. As a corollary, we show that $L(\mathbb{Z}^2\rtimes \{\pm I_2\})$ is the unique maximal Haagerup $G$-invariant von Neumann subalgebra in $L(G)$, where $I_2$ denotes the identity matrix in $SL_2(\mathbb{Z})$.
\end{abstract}



\maketitle

%\tableofcontents

\section{introduction}

Let $G$ be a countable discrete group and $L(G)$ be the corresponding group von Neumann algebra. Note that $G$ acts on $L(G)$ naturally by conjugation.
After the initial work by Alekseev-Brugger \cite{ab}, the problem of classifying $G$-invariant von Neumann subalgebras in $L(G)$, i.e., those which are globally invariant as a subset under the $G$-conjugation action, has quickly received  quite much attention recently \cites{cd,kp,aj,cds,aho,jz}. There are at least two potential applications behind this line of research. One is that every $G$-invariant von Neumann subalgebra in $L(G)$ is clearly regular, i.e., its normalizers generate the whole ambient von Neumann algebra $L(G)$. Thus, maximal abelian $G$-invariant von Neumann subalgebras are automatically Cartan subalgebras in $L(G)$. Hence the study of classifying $G$-invariant von Neumann subalgebras in $L(G)$ may shed light on the famous open question on the existence of Cartan subalgebras in $L(G)$ for certain classes of groups \cite[Problem V]{ioana_icm}, e.g., the class of icc groups with positive first $L^2$-Betti numbers \cite[Corollary 4.4]{aj} or lattices in higher rank simple Lie groups, e.g., $G=SL_3(\mathbb{Z})$ \cite[Corollary 2.8]{kp}. The other one comes from the problem of classifying the so-called invariant random von Neumann subalgebras, a new concept as introduced in \cite{aho}. This notion may be considered as the von Neumann algebra counterpart for the quite hot topic of invariant random subgroups \cites{agv,gel}. In fact, it is 
not hard to see, and has been partially indicated in \cite{aho,jz} that the concept of invariant random von Neumann subalgebras serves as a middle link between  invariant random subgroups and characters on groups \cite{bd}. Moreover, invariant von Neumann subalgebras are the simplest, i.e., the Dirac type invariant random von Neumann subalgebras.

For the above classification problem, an extreme situation is that every $G$-invariant von Neumann subalgebra $P$ in $L(G)$ satisfies that $P=L(H)$ for some normal subgroup $H\lhd G$. Once this happens, we say that $G$ has the invariant von Neumann subalgebras rigidity property (ISR property, for short) following Amrutam-Jiang in \cite{aj}. It is clear that by Pontryagin duality, infinite abelian groups do not have the ISR property. In fact, it was proved in \cite[Proposition 3.1]{aj} that if an infinite group $G$ is not icc (i.e., does not satisfy the infinite-conjugacy-class condition), then $G$ does not have the ISR property. On the other hand, many icc groups with trivial amenable radical are known to have the ISR property, including irreducible lattices in higher rank simple Lie groups \cite{kp}, 
 non-abelian free groups and a finite direct sum of them \cite{aj}, all acylindrically hyperbolic groups with trivial amenable radical  \cite{cds}, etc.  Recently, the (amenable) finitary permutation group
 $S_{\mathbb{N}}$ was proved to have this ISR property in \cite{jz}. Besides this extreme situation, it seems nothing is known on classifying invariant von Neumann subalgebras inside non-abelian ambient von Neumann algebras. Therefore,
it is natural to ask whether we can classify all $G$-invariant von Neumann subalgebras in $L(G)$ if $G$ does not have the ISR property.

Recall that in \cite[Example 3.5]{aj}, Amrutam and the first named author presented an example showing that i.c.c. condition is not yet sufficient for deducing the ISR property. More precisely, it was proved that for $G=\mathbb{Z}^2\rtimes SL_2(\mathbb{Z})$, there is a $G$-invariant von Neumann subalgebra $M\subsetneq L(\mathbb{Z}^2)$ such that $M$ is not equal to $L(H)$ for any normal subgroup $H$ in $G$. This ambient group $G$ and the associated action $SL_2(\mathbb{Z})\curvearrowright \mathbb{T}^2=\widehat{\mathbb{Z}^2}$ have played prominent roles in the modern development of von Neumann algebras and measurable group theory. In \cite{popa_betti}, Popa proved that $L(G)$ has trivial fundamental group, hence solving a long standing problem proposed by R. V. Kadison. This was also one of the early achievements of Popa's highly influential deformation/rigidity technique, see \cites{popa_icm, ioana_icm, vaes_icm} for an overview. Ozawa proved that $L(G)$ is solid \cite{ozawa_hokkaido} and hence it is prime in the sense that it can not be decomposed as a tensor product of two II$_1$ factors. In \cite{ioana_subequivalence}, Ioana established an alternative principle for all ergodic subequivalence relations inside the equivalence relation defined by $SL_2(\mathbb{Z})\curvearrowright \mathbb{T}^2$. He also  
proved that $L(G)$ has a unique group measure space Cartan subalgebra in \cite{ioana_gafa}. Shortly, this result was improved to be a unique Cartan by Popa and Vaes in \cite{pv_crelle}.  Recently, as a typical case of studying the notion of maximal Haagerup property as initiated in \cite{jiangskalski}, the first named author showed that $L(SL_2(\mathbb{Z}))$ is a maximal Haagerup von Neumann subalgebra in $L(G)$. For more recent study of the maximal Haagerup aspects and generalizations of the above results, see \cites{val, jv}.

Concerning the importance of studying various structure properties of $L(\mathbb{Z}^2\rtimes SL_2(\mathbb{Z}))$, it is natural to wonder whether we can classify all invariant von Neumann subalgebras in it. In this paper, we answer this question positively.

\begin{thm}\label{thm}
Let $G=\mathbb{Z}^2\rtimes SL_2(\mathbb{Z})$. Then a von Neumann subalgebra $P\subseteq L(G)$ is $G$-invariant if and only if either $P=L(H)$ for some normal subgroup $H\subseteq G$ or $P=A_n$ for some $n\geq 0$, where $A_n:=\{x\in L(n\mathbb{Z}^2):~\tau(xu_g)=\tau(xu_{g^{-1}}), ~\forall~g\in G\}$, where $\tau$ denotes the canonical trace on $L(G)$ defined by $\tau(x)=\langle x\delta_e, \delta_e\rangle$ for any $x\in L(G)\subseteq B(\ell^2(G))$. 
\end{thm}
In \cite{aho}, the authors proved the striking result that for any countable discrete group $G$, $L(G_a)$ is the maximal amenable $G$-invariant von Neumann subalgebra in $L(G)$, where  $G_a$ denotes the amenable radical in $G$, i.e., the unique maximal amenable normal subgroup in $G$.
As a natural generalization of amenable radicals for groups, the notion of Haagerup radical was considered in \cite{jiangskalski}. However, it is still unclear whether every countable group admits the Haagerup radical (see e.g., \cite[Question 2.11]{jiangskalski}), although this radical does exist for several classes of groups, see e.g., \cite[Proposition 2.10]{jiangskalski} and \cite[Proposition 4.1]{val}. Similarly, as suggested by Amrutam, one may also consider the notion of maximal Haagerup $G$-invariant von Neumann subalgebras in $L(G)$. Although the existence of such subalgebras in $L(G)$ (for $G$ without the Haagerup property) is still unclear in general, we deduce the following parallel result to \cite[Theorem A]{aho} but for a particular ambient group. To the best knowledge of the authors, this is the first known result along this direction. 
\begin{cor}\label{cor}
Let $G=\mathbb{Z}^2\rtimes SL_2(\mathbb{Z})$. Then $L(\mathbb{Z}^2\rtimes \{\pm I_2\})$ is the unique maximal Haagerup $G$-invariant von Neumann subalgebra in $L(G)$, where $I_2$ denotes the identity matrix in $SL_2(\mathbb{Z})$.
\end{cor}

Next, let us comment on the proof of Theorem \ref{thm}.
For the proof, we combine \cite[Theorem 5.1]{cds} with \cite[Theorem A]{aho} to  reduce the study to the case $P\subseteq L(\mathbb{Z}^2\rtimes \{\pm I_2\})$. Then by applying techniques as introduced in \cite{aj}, a considerable amount of effort has to be devoted to showing that $P$ actually lies in $L(\mathbb{Z}^2)$ unless $P=L(d\mathbb{Z}^2\rtimes \{\pm I_2\})$, where $d\in\{1, 2\}$.
We remark that a similar strategy can also be applied to the wreath product group $\mathbb{Z}\wr F_2$, see Section \ref{section: last remark}.

This paper is organized as follows: after recalling  relevant techniques in Section \ref{section: preliminaries}, we present the proof of Theorem \ref{thm} and Corollary \ref{cor} in Section \ref{section: proof}. Finally, in Section \ref{section: last remark}, we collect some remarks and open questions related to this work. 

In this paper, we usually use $\tau$ to mean a trace on a finite von Neumann algebra, e.g., a crossed product or group von Neumann algebra, which will be clear from the context.
\medskip

\paragraph{\textbf{Acknowledgements}} 

This work is partially supported by National Natural Science Foundation of China (Grant No. 12001081, No. 12271074) and the Fundamental Research Funds for the Central Universities (Grant No. DUT19RC(3)075). Part of this work was done during the visiting of Y. J. to the Institute for Advances Study in Mathematics of HIT in October, 2023. Y. J. is grateful to Prof. Simeng Wang for his invitation and hospitality during this visiting and to Prof. Adam Skalski and Dr. Amrutam Tattwamasi for helpful discussions on this paper.

\section{Preliminaries}\label{section: preliminaries}

In this section, we collect relevant techniques needed for proving Theorem \ref{thm}.

Let $(A,\phi)$ be a tracial von Neumann subalgebra equipped with a trace $\phi$ and $H\curvearrowright (A,\phi)$ be a $\phi$-preserving action. Then we may form the crossed product von Neumann algebra $A\rtimes H$ \cite[\S~5.2]{ap}, which is equipped with a trace $\tau$ defined by 
\[\tau(\sum_ha_hu_h)=\phi(a_e),~\text{where}~ \sum_ha_hu_h\in A\rtimes H.\]
In this paper, we are mainly interested in the case $A=L(\mathbb{Z}^2)$ equipped with the canonical trace and $H=SL_2(\mathbb{Z})$. By Pontrygain duality, $A\cong L^{\infty}(\mathbb{T}^2,\mu)$, where $\mu$ denotes the Haar measure on the 2-torus $\mathbb{T}^2$. The action $G\curvearrowright A$ is induced by the canonical left matrix multiplication on column vectors: $SL_2(\mathbb{Z})\curvearrowright \mathbb{Z}^2$. It is easy to check that $L(\mathbb{Z}^2\rtimes SL_2(\mathbb{Z}))\cong L^{\infty}(\mathbb{T}^2,\mu)\rtimes SL_2(\mathbb{Z})$.  

The following proposition is standard, but we decide to include its proof for completeness. 

\begin{prop}\label{prop: detect crossed product}
Let $H$ be any countable discrete group and $H\curvearrowright (A,\phi)$ be a $\phi$-preserving action on a tracial von Neumann algebra $(A,\phi)$.
Let $L(H)\subseteq P\subseteq A\rtimes H$ be an inclusion of von Neumann algebras. Denote by $E: A\rtimes H\twoheadrightarrow P$ the trace preserving conditional expectation onto $P$. Then $P=B\rtimes H$ for some $H$-invariant von Neumann subalgebra $B\subseteq A$ if and only if $E(A)\subseteq A$.
\end{prop}

\begin{proof}
The ``only if" direction is clear, and we are left to check the ``if" direction holds true. Define $B=E(A)$. The assumption $E(A)\subseteq A$ implies $B$ is a von Neumann subalgebra in $A\cap P$. Then we just need to check that $P=B\rtimes H$. On the one hand, it is clear that $B\rtimes H\subseteq P$. On the other hand, for any $a\in A$ and $h\in H$, we have $E(au_h)=E(a)u_h$. Then for any $x\in P$, by taking a sequence of finite sum $\sum_ha_h^{(n)}u_h$ with uniformly bounded norm to approximate $x$ in the strong operator topology, we deduce that $x=E(x)$ can be approximated by $E(\sum_ha_h^{(n)}u_h)=\sum_hE(a_h^{(n)})u_h\in B\rtimes H$. Hence $x\in B\rtimes H$. 
\end{proof}

For the convenience of later use, we record the following fact on $G=\mathbb{Z}^2\rtimes SL_2(\mathbb{Z})$. We refer to \cites{cs, cds} for unexplained notions.

\begin{prop}\label{prop: G as in chifan-das-sun}
The group $G=\mathbb{Z}^2\rtimes SL_2(\mathbb{Z})$ is an icc exact group which satisfies condition 2) in \cite[Theorem 5.1]{cds}, i.e., it admits a proper array into a weakly $\ell^2$-representation.
\end{prop}
\begin{proof}
It is easy to check that $G$ is icc.  Bi-exactness of $G$ was first proved by Ozawa \cite{ozawa_hokkaido}. Note that it is often denoted as the class $\mathcal{S}$ in the literature for Ozawa's class of bi-exact groups, e.g., \cite{ozawa_imrn}, \cite[Remark 19]{sako}. Hence from \cite[Definition 1.6 and Proposition 1.11]{cs}, we deduce that $G$ satisfies the above mentioned condition 2) in \cite[Theorem 5.1]{cds}, see also \cite[Proposition 1.9]{cs}.
\end{proof}

\section{Proof of Theorem \ref{thm} and Corollary \ref{cor}}\label{section: proof}

We are ready for the proof of Theorem \ref{thm}. 

The ``if" direction in Theorem \ref{thm} is easily verified, hence we just need to prove the ``only if" direction, which we record as a proposition.

\begin{prop}\label{prop: version of thm}
Let $G=\mathbb{Z}^2\rtimes SL_2(\mathbb{Z})$. Let $P$ be a $G$-invariant von Neumann subalgebra in $(L(G),\tau)$. Then 
\begin{itemize}
\item either $P=L(H)$ for some normal subgroup $H\lhd G$; or,
\item $P=A_n$ for some $n\geq 0$, where $A_n\subset L(n\mathbb{Z}^2)$ is defined by $A_n=\{x\in L(n\mathbb{Z}^2): \tau(xs)=\tau(xs^{-1}), \forall s\in n\mathbb{Z}^2\}$.
\end{itemize}
\end{prop}
\begin{proof}
First, note that $G$ satisfies condition 2) in \cite[Theorem 5.1]{cds} as explained in Proposition \ref{prop: G as in chifan-das-sun}. Therefore, we may apply \cite[Theorem 5.1]{cds} to deduce that if $P$ is non-amenable, then $P=L(H)$ for some normal subgroup $H\lhd G$.

From now on, we may assume that $P$ is amenable. Denote by $\mathbb{Z}/{2\mathbb{Z}}=\langle s \rangle$, where $s=-I_2\in SL_2(\mathbb{Z})$. 

Note that the amenable radical of $G$ is $\mathbb{Z}^2\rtimes \mathbb{Z}/{2\mathbb{Z}}$, say e.g., by \cite[Proposition 2.10]{jiangskalski}. 
From \cite[Theorem A]{aho}, we conclude that $P\subseteq L(\mathbb{Z}^2\rtimes \mathbb{Z}/{2\mathbb{Z}})$. Therefore, it suffices to prove the following claim. 

\textbf{Claim}: if $P\subseteq L(\mathbb{Z}^2\rtimes \mathbb{Z}/{2\mathbb{Z}})$ and $P$ is $G$-invariant, then $P\in\{\mathbb{C}, L(n\mathbb{Z}^2)(n>0), A_n(n>0), L(2\mathbb{Z}^2\rtimes \mathbb{Z}/{2\mathbb{Z}}), L(\mathbb{Z}^2\rtimes \mathbb{Z}/{2\mathbb{Z}})\}$.
\begin{proof}[Proof of the \textbf{Claim}]
Let $E: L(\mathbb{Z}^2\rtimes \mathbb{Z}/{2\mathbb{Z}})\rightarrow P$ be the trace $\tau$-preserving conditional expectation onto $P$. Below, we directly write $g$ for the canonical unitary $u_g$ inside $L(G)$ for simplicity and we use the notation $\langle a, b \rangle$ to mean $\tau(b^*a)$  for any $a, b\in L(G)$.

Note that $L(SL_2(\mathbb{Z}))'\cap L(\mathbb{Z}^2\rtimes SL_2(\mathbb{Z}))=L(\mathbb{Z}/{2\mathbb{Z}})$, say by a direct calculation using \cite[Lemma 2.7]{aj}.

Since $P$ is $G$-invariant, we deduce that $gE(s)g^{-1}=E(gsg^{-1})=E(s)$ for all $g\in SL_2(\mathbb{Z})$. Hence, 
$E(s)\in L(SL_2(\mathbb{Z}))'\cap L(\mathbb{Z}^2\rtimes SL_2(\mathbb{Z}))=L(\mathbb{Z}/{2\mathbb{Z}})$. Therefore, we may write $E(s)=\lambda+\mu s$, where $\lambda,\mu\in\mathbb{C}$. Observe that $\lambda=\tau(E(s))=\tau(s)=0$. Hence, $\mu^2=(\mu s)^2=E(s)E(s)=E(sE(s))=E(s\mu s)=\mu$. It follows that $\mu=0$ or $1$. Hence, $E(s)=0$ or $s$.

\textbf{Case 1}: $E(s)=s$.

Take any $v\in\mathbb{Z}^2$, we have $vsv^{-1}\in P$, i.e., $(v\sigma_s(v^{-1}))s\in P$, so $v\sigma_s(v^{-1})\in P$. Notice that if we write $v=(x, y)^t$, where the superscript ``$t$" stands for the transpose of the row vector $(x, y)$, then $v\sigma_s(v^{-1})=2(x,y)^t$. Hence $L(2\mathbb{Z}^2\rtimes \mathbb{Z}/{2\mathbb{Z}})\subseteq P$.

Let us show that either $P=L(2\mathbb{Z}^2\rtimes \mathbb{Z}/{2\mathbb{Z}})$ or $P=L(\mathbb{Z}^2\rtimes \mathbb{Z}/{2\mathbb{Z}})$.

Notice that for any $a\in L(\mathbb{Z}^2)$, we have $E(a)\in L(\mathbb{Z}^2)'\cap L(G)=L(\mathbb{Z}^2)$.
Thus, $P=E(L(\mathbb{Z}^2))\rtimes \mathbb{Z}/{2\mathbb{Z}}$ by Proposition \ref{prop: detect crossed product}. Therefore, $E(L(\mathbb{Z}^2))=P\cap L(\mathbb{Z}^2)$ and it is clearly $SL_2(\mathbb{Z})$-invariant. By 
\cite[Example 5.9]{wit} (which is essentially a minor correction of \cite[Theorem 2.3]{park}, see also the proof
of \cite[Lemma 3.5]{jiangskalski}), we deduce that either $P\cap L(\mathbb{Z}^2)=L(n\mathbb{Z}^2)$ or $P\cap L(\mathbb{Z}^2)=A_n$ for some $n\geq 0$. Here, $n\mathbb{Z}^2=n\mathbb{Z}\oplus n\mathbb{Z}$. Combining this with the fact that $L(2\mathbb{Z}^2)\rtimes \mathbb{Z}/{2\mathbb{Z}}\subseteq P$, we can deduce that $P\cap L(\mathbb{Z}^2)=L(n\mathbb{Z}^2)$ for $n=1$ or $2$; equivalently, $P=L(\mathbb{Z}^2\rtimes \mathbb{Z}/{2\mathbb{Z}})$ or $L(2\mathbb{Z}^2\rtimes \mathbb{Z}/{2\mathbb{Z}})$.

\textbf{Case 2}: $E(s)=0$.

We record the following observation for later use.

\textbf{Observation}: If $E(s)=0$, then $E(vs)=0$ for all $v\in\mathbb{Z}^2\Leftrightarrow E(e_1s)=0$, where $e_1=(1, 0)^t\in\mathbb{Z}^2$.

\begin{proof}[Proof of the Observation]
We just need to check $\Leftarrow$ holds. First, for any $v=(x, 0)^t\in\mathbb{Z}^2$, we get that $0=vE(s)v^{-1}=E(vsv^{-1})=E(v\sigma_s(v^{-1})s)=E((2x, 0)^ts)$; similarly, $0=vE(e_1s)v^{-1}=E((2x+1, 0)^ts)$. Since $x\in\mathbb{Z}$ is arbitrary, we deduce that $E(e_ns)=0$ for all $n\geq 1$, where $e_n=(n, 0)^t$. Next, take any $v=(x, y)^t\in\mathbb{Z}^2\setminus\{0\}$, set $n=\text{gcd}(|x|, |y|)$. Here, if $|x|=0$ (respectively $|y|=0$), then $n:=|y|$ (respectively $n:=|x|$). Observe that there exists some $g\in SL_2(\mathbb{Z})$ such that $v=g\cdot e_n$, where $g\cdot e_n$ denotes the matrix left multiplication. Indeed, from $n=\gcd(|x|, |y|)$, we get two integers $a, b\in\mathbb{Z}$ such that $xa-yb=n$, then set $g:=\begin{pmatrix}
\frac{x}{n}&b\\
\frac{y}{n}&a
\end{pmatrix}$.

Hence $E(vs)=E((g\cdot e_n)s)=E(ge_ng^{-1}s)=E(g(e_ns)g^{-1})=gE(e_ns)g^{-1}=0$.
\end{proof}

Since $P\cap L(\mathbb{Z}^2)$ is $SL_2(\mathbb{Z})$-invariant, we may apply \cite[Example 5.9]{wit} to deduce that $P\cap L(\mathbb{Z}^2)=\mathbb{C}$, $L(n\mathbb{Z}^2)$ or $A_n$ for some $n\geq 1$. We split the proof by considering three subcases.

\textbf{Subcase I}: $P\cap L(\mathbb{Z}^2)=\mathbb{C}$.

We claim that $P=\mathbb{C}$.

First, observe that for any $v\in \mathbb{Z}^2$, we have 
\begin{align*}
E(v)\in P\cap L(\mathbb{Z}^2)'=P\cap L(\mathbb{Z}^2)=\mathbb{C},
\end{align*}
where to get the 2nd equality, we have used the fact that $L(\mathbb{Z}^2)$ is a masa in $L(G)$.
Thus, $\forall v\in \mathbb{Z}^2\setminus \{0\}$, we get $E(v)=\tau(E(v))=\tau(v)=0$.

We are left to show $E(vs)=0$ for all $v\in\mathbb{Z}^2\setminus \{0\}$. By the above observation, it suffices to show that $E(e_1s)=0$. 

First, notice that
\begin{align*}
E(e_1s)\in L(\{\begin{pmatrix}
1&\mathbb{Z}\\
0&1
\end{pmatrix} \})'\cap L(\mathbb{Z}^2\rtimes \mathbb{Z}/{2\mathbb{Z}})\subseteq L((\mathbb{Z}, 0)^t)\rtimes \mathbb{Z}/{2\mathbb{Z}}.
\end{align*}
Thus, we may write $E(e_1s)=a+bs$, where $a, b\in L((\mathbb{Z}, 0)^t)$.

From $\langle v-E(v), E(e_1s)\rangle=0$, we get that $\langle v, a \rangle=0$ for all $v\in\mathbb{Z}^2\setminus \{0\}$. Hence, $a\in\mathbb{C}$. Then by computing the tace of $E(e_1s)$, we get that $a=\tau(a+bs)=\tau(E(e_1s))=\tau(e_1s)=0$. Hence, $E(e_1s)=bs$. 

Let us write $b=\sum_{n\in\mathbb{Z}}\mu_ne_n$, where $\mu_n\in\mathbb{C}$ and set $f_n=(0, n)^t\in\mathbb{Z}^2$. 
Notice that $f_n=\begin{pmatrix}
0&-1\\
1&0
\end{pmatrix}\cdot e_n$. Thus,

$E(f_1s)=E(\begin{pmatrix}
0&-1\\
1&0
\end{pmatrix}e_1\begin{pmatrix}
0&1\\
-1&0
\end{pmatrix}s)=\begin{pmatrix}
0&-1\\
1&0
\end{pmatrix} E(e_1s)\begin{pmatrix}
0&1\\
-1&0
\end{pmatrix}=\sum_{n\in\mathbb{Z}}\mu_n f_ns$. 

By $P$-bimodule property of $E$, we have $E(e_1s)E(f_1s)=E(e_1sE(f_1s))$. Let us compute both sides concretely.
\begin{align*}
E(e_1s)E(f_1s)&=\sum_{n, m\in\mathbb{Z}}\mu_m\mu_n\begin{pmatrix}
m\\n
\end{pmatrix} ,\\
E(e_1sE(f_1s))&=E(e_1s(\sum_{m\in\mathbb{Z}}\mu_mf_ms))=\sum_{m\in\mathbb{Z}}\mu_mE(\begin{pmatrix}
1\\
-m
\end{pmatrix})=0,
\end{align*}
where to get the last equality, we used the fact that $E(v)=0$ for all $v\in\mathbb{Z}^2\setminus\{0\}$.
Hence, $\mu_m\mu_n=0$ for all $(m, n)\in\mathbb{Z}^2$. Thus, $\mu_n^2=0$, i.e., $\mu_n=0$ for all $n\in\mathbb{Z}$; equivalently, $b=0$ and thus $E(e_1s)=0$. The proof of this subcase is done.

\textbf{Subcase II}: $P\cap L(\mathbb{Z}^2)=L(n\mathbb{Z}^2)$ for some $n\geq 1$.

We claim that $P=L(n\mathbb{Z}^2)$. 

If $n=1$, then $L(\mathbb{Z}^2)\subseteq P$ and thus $E(v)=v$ for all $v\in\mathbb{Z}^2$. Hence  $E(vs)=vE(s)=0$. Thus, $P=L(\mathbb{Z}^2)$.

From now on, we assume that $n\geq 2$.
The proof given below is essentially the same as the proof of subcase I with minor modification, we record it for completeness.

First, for any $v\in\mathbb{Z}^2\setminus {n\mathbb{Z}^2}$, we observe that $E(v)=0$.

Indeed, 
\begin{align*}
E(v)\in P\cap L(\mathbb{Z}^2)'=P\cap L(\mathbb{Z}^2)=L(n\mathbb{Z}^2).
\end{align*}
Thus, from $\langle v-E(v), E(v)\rangle=0$, we deduce that $\langle E(v), E(v) \rangle=0$, i.e., $E(v)=0$.

Thus, for any $v\in \mathbb{Z}^2$, we have either $E(v)=0$ or $v\in n\mathbb{Z}^2$ and in this case $E(v)=v$.

We are left to show $E(vs)=0$. It suffices to show $E(e_1s)=0$ by the above observation.

First, notice that
\begin{align*}
E(e_1s)\in L(\{\begin{pmatrix}
1&\mathbb{Z}\\
0&1
\end{pmatrix} \})'\cap L(\mathbb{Z}^2\rtimes \mathbb{Z}/{2\mathbb{Z}})\subseteq L((\mathbb{Z}, 0)^t)\rtimes \mathbb{Z}/{2\mathbb{Z}}.
\end{align*}
Thus, we may write $E(e_1s)=a+bs$, where $a, b\in L((\mathbb{Z}, 0)^t)$.

From $\langle v-E(v), E(e_1s)\rangle=0$, we deduce that $\langle v, a \rangle=0$ for all $v\in \mathbb{Z}^2\setminus {n\mathbb{Z}^2}$. Hence, $a\in L((n\mathbb{Z},0)^t)\subset L(n\mathbb{Z}^2)\subseteq P$. Similarly, from $\langle e_1s-E(e_1s), a \rangle=0$, we deduce that $\langle a, a \rangle=0$, i.e., $a=0$.
Hence, $E(e_1s)=bs$.

Then the last part of the proof is exactly the same as the proof of subcase I.

\textbf{Subcase III}: $P\cap L(\mathbb{Z}^2)=A_n$ for some $n\geq 1$.

We claim that $P=A_n$. 

First, fix any non-zero vector $v\in n\mathbb{Z}^2$, we have $E(v)\in P\cap L(\mathbb{Z}^2)'=P\cap L(\mathbb{Z}^2)=A_n$. Hence, we can write $E(v)=\sum\limits_{\omega\in n\mathbb{Z}^2}\lambda_{\omega}\omega$. From $0=\langle v-E(v), P \rangle$, we deduce that $\langle v-\sum\limits_{\omega\in n\mathbb{Z}^2}\lambda_{\omega}\omega, \omega_0+\omega_{0}^{-1} \rangle=0$ for all $\omega_0\in n\mathbb{Z}^2$. Observe that by taking $\omega_0\not\in\{v, v^{-1}\}$, we can get that $0=\lambda_{\omega_0}+\lambda_{\omega_0^{-1}}=\lambda_{\omega_0}.$ Similarly, by taking $\omega_0=v$, we deduce that $\lambda_{v}=\lambda_{v^{-1}}=\frac{1}{2}$. Hence $E(v)=\frac{v+v^{-1}}{2}$ for all $v\in n\mathbb{Z}^2$.

Second, we check that $E(v)=0$ for all $v\in\mathbb{Z}^2\setminus {n\mathbb{Z}^2}$.

Observe that we still have $E(v)\in P\cap L(\mathbb{Z}^2)'=P\cap L(\mathbb{Z}^2)=A_n$. Hence, from the fact that $0=\langle v-E(v), E(v) \rangle$, we deduce that $\langle E(v), E(v)\rangle=0$ for all $v\in \mathbb{Z}^2\setminus {n\mathbb{Z}^2}$ since $\langle v, E(v) \rangle=0$ for such a $v$, thus $E(v)=0$ is proved.

We are left to show $E(e_1s)=0$. 

Once again, we still have 
that
\begin{align*}
E(e_1s)\in L(\{\begin{pmatrix}
1&\mathbb{Z}\\
0&1
\end{pmatrix} \})'\cap L(\mathbb{Z}^2\rtimes \mathbb{Z}/{2\mathbb{Z}})\subseteq L((\mathbb{Z}, 0)^t)\rtimes \mathbb{Z}/{2\mathbb{Z}}.
\end{align*}
Thus, we may write $E(e_1s)=a+bs$, where $a, b\in L((\mathbb{Z}, 0)^t)$.

From $\langle v-E(v), E(e_1s)\rangle=0$, we deduce that $\langle v, a \rangle=0$ for all $v\in \mathbb{Z}^2\setminus {n\mathbb{Z}^2}$. Hence, $a\in L((n\mathbb{Z},0)^t)$.

Next, from $\langle e_1s-E(e_1s), A_n \rangle=0$, we deduce that $\langle a, A_n \rangle=0$, equivalently, $a+\sigma_s(a)=0$. In other words, if we write $a=\sum_{i\in\mathbb{Z}}\lambda_ie_{ni}$, then 
\begin{align}\label{eq: skew-symmetry on the coefficients of a}
\lambda_i+\lambda_{-i}=0,\forall~i\in\mathbb{Z}. 
\end{align}

Now, let us write $a=\sum_{i\in\mathbb{Z}}\lambda_ie_{ni}$ and $b=\sum_{j\in\mathbb{Z}}\mu_je_j$, where $\lambda_i,\mu_j\in \mathbb{C}$ for all $i, j$.

Thus, $E(f_1s)=\begin{pmatrix}
0&-1\\
1&0
\end{pmatrix} E(e_1s)\begin{pmatrix}
0&1\\
-1&0
\end{pmatrix} =(\sum_{i\in\mathbb{Z}}\lambda_if_{ni})+(\sum_{j\in\mathbb{Z}}\mu_jf_j)s$.

Next, we compute both sides of the identity $E(e_1s)E(f_1s)=E(e_1sE(f_1s))$ concretely.

On the one hand, 
\begin{align*}
E(e_1s)E(f_1s)=[\sum_{i, j}\lambda_i\lambda_j\begin{pmatrix}
ni\\
nj
\end{pmatrix}+\sum_{j, k}\mu_j\mu_k\begin{pmatrix}
j\\
-k
\end{pmatrix}]+[\sum_{i,j}\lambda_i\mu_j\begin{pmatrix}
ni\\
j
\end{pmatrix}+\sum_{i, j}\lambda_i\mu_j\begin{pmatrix}
j\\
-ni
\end{pmatrix}]s;
\end{align*}
on the other hand, we have
\begin{align*}
E(e_1sE(f_1s))&=\sum_i\lambda_iE(\begin{pmatrix}
1\\
-ni
\end{pmatrix}s)+\sum_j\mu_jE(\begin{pmatrix}
1\\
-j
\end{pmatrix})\\
&=\sum_i\lambda_i\begin{pmatrix}
1&0\\
-ni&1
\end{pmatrix}E(e_1s)\begin{pmatrix}
1&0\\
ni&1
\end{pmatrix}+\sum_j\mu_jE(\begin{pmatrix}
1\\
-j
\end{pmatrix})\\
&=\sum_{i,j}\lambda_i\lambda_j\begin{pmatrix}
nj\\
-n^2ij
\end{pmatrix}+\sum_{i, k}\lambda_i\mu_k\begin{pmatrix}
k\\
-nik
\end{pmatrix}s+\sum_j\mu_jE(\begin{pmatrix}
1\\
-j
\end{pmatrix}).
\end{align*}

Due to the difference in computing $E(\begin{pmatrix}
1\\
-j
\end{pmatrix})$, we need to split the proof by considering two cases.

\textbf{Subsubcase I}: $n=1$.

In this case, it is routine to check that $E(\begin{pmatrix}
1\\
-j
\end{pmatrix})=\frac{1}{2}[\begin{pmatrix}
1\\
-j
\end{pmatrix}+\begin{pmatrix}
-1\\
j
\end{pmatrix}]$. Thus we may continue the above calculation to deduce that 
\begin{align*}
E(e_1sE(f_1s))&=\sum_{i,j}\lambda_i\lambda_j\begin{pmatrix}
j\\
-ij
\end{pmatrix}+\sum_{i, k}\lambda_i\mu_k\begin{pmatrix}
k\\
-ik
\end{pmatrix}s+\sum_j\mu_j\frac{1}{2}[\begin{pmatrix}
1\\
-j
\end{pmatrix}+\begin{pmatrix}
-1\\
j
\end{pmatrix}],\\
E(e_1s)E(f_1s)&=[\sum_{i, j}\lambda_i\lambda_j\begin{pmatrix}
i\\
j
\end{pmatrix}+\sum_{j, k}\mu_j\mu_k\begin{pmatrix}
j\\
-k
\end{pmatrix}]+[\sum_{i,j}\lambda_i\mu_j\begin{pmatrix}
i\\
j
\end{pmatrix}+\sum_{i, j}\lambda_i\mu_j\begin{pmatrix}
j\\
-i
\end{pmatrix}]s.
\end{align*}
Therefore, we can deduce that
\begin{align}
\label{eq-a: n=1 case}\sum_{i,j}\lambda_i\lambda_j\begin{pmatrix}
j\\
-ij
\end{pmatrix}+\sum_j\mu_j\frac{1}{2}[\begin{pmatrix}
1\\
-j
\end{pmatrix}+\begin{pmatrix}
-1\\
j
\end{pmatrix}]=\sum_{i, j}\lambda_i\lambda_j\begin{pmatrix}
i\\
j
\end{pmatrix}+\sum_{j, k}\mu_j\mu_k\begin{pmatrix}
j\\
-k
\end{pmatrix},\\
\label{eq-b: n=1 case}\sum_{i, k}\lambda_i\mu_k\begin{pmatrix}
k\\
-ik
\end{pmatrix}=\sum_{i,j}\lambda_i\mu_j\begin{pmatrix}
i\\
j
\end{pmatrix}+\sum_{i, j}\lambda_i\mu_j\begin{pmatrix}
j\\
-i
\end{pmatrix}.
\end{align}
Next, note that
\begin{align*}
E((e_1+e_{-1})s)&=E(e_1s)+E(e_{-1}s)=E(e_1s)+\sigma_s(E(e_1s))\\
&=(a+bs)+\sigma_s(a+bs)
=(a+\sigma_s(a))+(b+\sigma_s(b))s=(b+\sigma_s(b))s.
\end{align*}
Meanwhile, since $b+\sigma_s(b)\in A_1\subseteq P$, we deduce that $(b+\sigma_s(b))s=E((e_1+e_{-1})s)=E(E((e_1+e_{-1})s))=E((b+\sigma_s(b))s)=(b+\sigma_s(b))E(s)=0$, i.e., $b+\sigma_s(b)=0$, equivalently, 
\begin{align}\label{eq: skew-symmetry on coefficients of b, case n=1}
\mu_k+\mu_{-k}=0, \forall ~k\in \mathbb{Z}.
\end{align}

Now we can compare the coefficients of $\begin{pmatrix}
1\\
j
\end{pmatrix}$ and $\begin{pmatrix}
-1\\
j
\end{pmatrix}$ respectively on both sides of  \eqref{eq-a: n=1 case} to deduce that
\begin{align*}
\lambda_1\lambda_{-j}+\frac{1}{2}\mu_{-j}&=\lambda_1\lambda_j+\mu_1\mu_{-j},\\
\lambda_j\lambda_{-1}+\frac{1}{2}\mu_j&=\lambda_{-1}\lambda_j+\mu_{-1}\mu_{-j}.
\end{align*}
Take the sum of the above two equations and apply \eqref{eq: skew-symmetry on the coefficients of a} and \eqref{eq: skew-symmetry on coefficients of b, case n=1}, we get that $2\lambda_1\lambda_{-j}=0$ for all $j\in\mathbb{Z}$. Hence $\lambda_1=0$. Then plugging it in the above first equation to get that $\frac{1}{2}\mu_{-j}=\mu_1\mu_{-j}$. Thus either $\mu_1=\frac{1}{2}$ or $\mu_{-j}=0$ for all $j\in \mathbb{Z}$. 

Once we have $\mu_j=0$ for all $j\in\mathbb{Z}$, i.e., $b=0$, then $a=E(e_1s)=E(E(e_1s))=E(a)\in P\cap L(\mathbb{Z}^2)'=P\cap L(\mathbb{Z}^2)=A_1$. Recall that \eqref{eq: skew-symmetry on the coefficients of a} holds, i.e., $\langle a, A_1 \rangle=0$. Hence $a=0$; equivalently, $E(e_1s)=0$.
Therefore, we may assume $\mu_1=\frac{1}{2}$ and try to deduce a contradiction.

We compute the coefficient of $\begin{pmatrix}
-1\\
j
\end{pmatrix}$ on both sides of \eqref{eq-b: n=1 case} to get that
$\lambda_j\mu_{-1}=\lambda_{-1}\mu_j+\lambda_{-j}\mu_{-1}$. By plugging $\lambda_{-1}=-\lambda_1=0$, $\mu_{-1}=-\mu_1=-\frac{1}{2}$ in it, we get $\lambda_j=0$ for all $j\in \mathbb{Z}$, i.e., $a=0$. 

Then for any $|i|\neq 0, 1$, we compute the coefficients of $\begin{pmatrix}
i\\
j
\end{pmatrix}$ on both sides of \eqref{eq-b: n=1 case} to get that
$0=\lambda_i\lambda_j+\mu_i\mu_{-j}=\mu_i\mu_{-j}$. Therefore, $\mu_i=0$ for all $|i|\neq 0, 1$. Recall that $\mu_0+\mu_{-0}=0$, i.e., $\mu_0=0$, hence, we have shown that
$E(e_1s)=0+\frac{1}{2}(\begin{pmatrix}
1\\
0
\end{pmatrix}-\begin{pmatrix}
-1\\
0
\end{pmatrix})s$.

Then, $E(\begin{pmatrix}
3\\
0
\end{pmatrix}s)=\begin{pmatrix}
1\\
0
\end{pmatrix}E(e_1s)\begin{pmatrix}
-1\\
0
\end{pmatrix}=\frac{1}{2}(\begin{pmatrix}
3\\
0
\end{pmatrix}-\begin{pmatrix}
1\\
0
\end{pmatrix})s$.
Hence, 
\begin{align*}
P\ni E(\begin{pmatrix}
3\\
0
\end{pmatrix}s)E(e_1s)&=\frac{1}{4}[\begin{pmatrix}
3\\
0
\end{pmatrix}-\begin{pmatrix}
1\\
0
\end{pmatrix}][\begin{pmatrix}
-1\\
0
\end{pmatrix}-\begin{pmatrix}
1\\
0
\end{pmatrix}]\\
&=\frac{1}{4}[\begin{pmatrix}
2\\
0
\end{pmatrix}-1-\begin{pmatrix}
4\\
0
\end{pmatrix}+\begin{pmatrix}
2\\
0
\end{pmatrix}]\in L(\mathbb{Z}^2)\setminus A_1.
\end{align*}
This contradicts to the assumption that $P\cap L(\mathbb{Z}^2)=A_1$.

\textbf{Subsubcase II}: $n\geq 2$.

In this case, $\begin{pmatrix}
1\\
-j
\end{pmatrix}\not\in n\mathbb{Z}^2$ and hence $ E(\begin{pmatrix}
1\\
-j
\end{pmatrix})=0$, thus, we get the following identities by comparing the computation of $E(e_1s)E(f_1s)$ and $E(e_1sE(f_1s))$:
\begin{align}
\label{eq1: two sides, n>1} \sum_{i, j}\lambda_i\lambda_j\begin{pmatrix}
ni\\
nj
\end{pmatrix}+\sum_{j, k}\mu_j\mu_k\begin{pmatrix}
j\\
-k
\end{pmatrix}=\sum_{i,j}\lambda_i\lambda_j\begin{pmatrix}
nj\\
-n^2ij
\end{pmatrix},\\
\label{eq2: two sides, n>1}\sum_{i,j}\lambda_i\mu_j\begin{pmatrix}
ni\\
j
\end{pmatrix}+\sum_{i, j}\lambda_i\mu_j\begin{pmatrix}
j\\
-ni
\end{pmatrix}=\sum_{i, k}\lambda_i\mu_k\begin{pmatrix}
k\\
-nik
\end{pmatrix}.
\end{align}

By \eqref{eq1: two sides, n>1}, we deduce that $\mu_j\mu_k=0$ for all $(j,-k)^t\not\in n\mathbb{Z}^2$. In particular, $\mu_j=0$ for all $j\not\in n\mathbb{Z}$, hence, $b\in L((n\mathbb{Z}, 0)^t)$.

Next, note that
\begin{align*}
E((e_1+e_{-1})s)=E(e_1s)+E(e_{-1}s)=(a+bs)+\sigma_s(a+bs)\\
=(a+\sigma_s(a))+(b+\sigma_s(b))s=(b+\sigma_s(b))s.
\end{align*}
Meanwhile, since $b+\sigma_s(b)\in A_n\subseteq P$, we deduce that $(b+\sigma_s(b))s=E((e_1+e_{-1})s)=E(E((e_1+e_{-1})s))=E((b+\sigma_s(b))s)=(b+\sigma_s(b))E(s)=0$, i.e., $b+\sigma_s(b)=0$, equivalently, 
\begin{align}\label{eq: skew-symmetry on coefficients of b}
\mu_k+\mu_{-k}=0, \forall ~k\in \mathbb{Z}.
\end{align}

For any $i\neq 0$ and $j\in\mathbb{Z}$, by comparing the coefficients of $\begin{pmatrix}
ni\\
nj
\end{pmatrix}$ on both sides of the two identities \eqref{eq1: two sides, n>1} and \eqref{eq2: two sides, n>1}, we deduce that 
\begin{align}
\label{eq11: simplied eq1} \lambda_i\lambda_j+\mu_{ni}\mu_{-nj}&=\lambda_{-\frac{j}{ni}}\lambda_i, \forall~ i\neq 0, \forall ~j\\
\label{eq22: simplied eq2} \lambda_i\mu_{nj}+\lambda_{-j}\mu_{ni}&=\lambda_{-\frac{j}{ni}}\mu_{ni}, \forall~i\neq 0,\forall~j.
\end{align}
Here, $\lambda_{-\frac{j}{ni}}$ is understood as 0 if $(ni)\nmid j$.

Substitute $j=1$ into \eqref{eq11: simplied eq1} and 
\eqref{eq22: simplied eq2} and use \eqref{eq: skew-symmetry on coefficients of b} to deduce that 
\begin{align}
\label{eq3: j=1 case} \lambda_i\lambda_1=\mu_{ni}\mu_n, \forall~ i\neq 0,\\
\label{eq4: j=1 case} \lambda_i\mu_n=\lambda_1\mu_{ni},\forall~i\neq 0.
\end{align}
This implies that $(\lambda_i^2-\mu_{ni}^2)\lambda_1\mu_n=0$ and $\lambda_1^2=\mu_n^2$ (by plugging $i=1$ in \eqref{eq3: j=1 case}). 

By plugging $j=nik$ in \eqref{eq11: simplied eq1} and \eqref{eq22: simplied eq2} and using \eqref{eq: skew-symmetry on the coefficients of a}, we get that
\begin{align}
\label{eq5} \lambda_i\lambda_{nik}-\mu_{ni}\mu_{n^2ik}=\lambda_{-k}\lambda_i, \forall ~i\neq 0, \forall~ k, \\
\label{eq6} \lambda_i\mu_{n^2ik}-\lambda_{nik}\mu_{ni}=\lambda_{-k}\mu_{ni},\forall~i\neq 0, \forall~k.
\end{align}

Claim 1: $\lambda_1\mu_n=0$; equivalently, $\lambda_1=0=\mu_n$ (since $\lambda_1^2=\mu_n^2$). 

Indeed, assume not, i.e., $\lambda_1=\pm \mu_n\neq 0$. We can deduce from \eqref{eq3: j=1 case} that $\lambda_i=\pm \mu_{ni}$. In both cases, we can deduce from \eqref{eq5} that $\lambda_{-k}\lambda_i=0$ for all $i\neq 0$ and $k\in\mathbb{Z}$. In particular, $\lambda_i=0$ for all $i\neq 0$, contradicting to the assumption that $\lambda_1\neq 0$.

Claim 2: $\lambda_i=\mu_{ni}=0$ for all $i\neq 0$.

First, let us check that $\lambda_i-\mu_{ni}=0$ for all $i\neq 0$. Assume this does not hold, then for some $i\neq 0$, we have $\lambda_i-\mu_{ni}\neq 0$. By setting $j=i$ in \eqref{eq11: simplied eq1}, we conclude that $\lambda_i^2-\mu_{ni}^2=0$. Hence, we have $\lambda_i+\mu_{ni}=0$ and thus $\lambda_i=-\mu_{ni}$. Then using this relation, we may deduce from \eqref{eq5} and \eqref{eq6} that
$\lambda_i(\lambda_{nik}+\mu_{n^2ik})=\lambda_{-k}\lambda_i$ and $\lambda_i(\lambda_{nik}+\mu_{n^2ik})=-\lambda_{-k}\lambda_i$. Thus
$0=\lambda_{-k}\lambda_i$ for all $k\in\mathbb{Z}$, thus $\lambda_i=0$ and $\mu_{ni}=-\lambda_i=0$, contradicting to our assumption that $\lambda_i-\mu_{ni}\neq 0$. Hence, we have proved that $\lambda_i=\mu_{ni}$ for all $i\neq 0$. Then, it follows from \eqref{eq5} that $\lambda_{-k}\lambda_i=0$ for all $i\neq 0$, thus $\lambda_i=0$ for all $i\neq 0$.

Based on Claim 2 and the fact that $\mu_j=0$ for all $j\not\in n\mathbb{Z}$, we deduce that $a, b\in\mathbb{C}$, then by taking trace on $E(e_1s)=a+bs$, we get $a=0$. By taking $E$ on $E(e_1s)=bs$, we get $bs=E(bs)=bE(s)=0$, i.e., $b=0$, and hence $E(e_1s)=0$. 
\end{proof}
This finishes the proof of Proposition \ref{prop: version of thm} and hence also of Theorem \ref{thm}.
\end{proof}

Finally, let us prove Corollary \ref{cor}.
\begin{proof}[Proof of Corollary \ref{cor}]
Let $P$ be a $G$-invariant von Neumann subalgebra in $L(G)$ with the Haagerup property. Then $P=A_n$ for some $n\geq 0$ or $L(H)$ for some normal subgroup $H\lhd G$ with the Haagerup property by Theorem \ref{thm}. By \cite[Proposition 2.10]{jiangskalski}, we know that $H\subseteq \mathbb{Z}^2\rtimes \{\pm I_2\}$, where $I_2$ denotes the identity matrix in $SL_2(\mathbb{Z})$. Therefore, $P\subseteq L(\mathbb{Z}^2\rtimes \{\pm I_2\})$ in both cases.
Notice that $L(\mathbb{Z}^2\rtimes \{\pm I_2\})$ is clearly $G$-invariant and has Haagerup property and hence it is the maximal one with these properties.
\end{proof}

\section{Remarks and open questions}\label{section: last remark}

In this section, we record some remarks on the strategy used in this paper and open questions which might worth further investigating.

Note that $G=\mathbb{Z}\wr F_2$ is also a bi-exact group by \cite[Corollary 15.3.9]{bo} with amenable radical being $\oplus_{F_2}\mathbb{Z}$, hence for a  $G$-invariant von Neumann subaglebra $P$ in $L(G)$, we can also apply the same strategy as used in this paper to deduce that either $P=L(H)$ for some non-amenable normal subgroup $H\subseteq G$ or $P\subseteq L(\oplus_{F_2}\mathbb{Z})$ is $F_2$-invariant; equivalently, $P=L^{\infty}(Y,\nu)$ for some quotient action $F_2\curvearrowright (Y,\nu)$ of the Bernoulli shift $F_2\curvearrowright (\mathbb{T}^{F_2},\mu^{F_2})$. However, it seems impossible to explicitly describe all the quotient actions $F_2\curvearrowright (Y,\nu)$.

Let us list some open questions related to this paper.

The first is to identify more icc groups $G$ without the ISR property such that all $G$-invariant von Neumann subalgebras in $L(G)$ can be classified. In particular, we mention the following questions.

\begin{question}[suggested by Amrutam]
Let $n\geq 3$. Set $G=\mathbb{Z}^n\rtimes SL_n(\mathbb{Z})$ and $M=L(G)$. Classify all $G$-invariant von Neumann subalgebras in $M$.
\end{question}
\begin{question}
In \cite{val}, Valette studied a very nice generalization of the group $\mathbb{Z}^2\rtimes GL_2(\mathbb{Z})=:G_1$ to $G_n:=\mathbb{Z}^{n+1}\rtimes_{\rho_n}GL_2(\mathbb{Z})$ and proved that all $G_n$'s have similar classifications on maximal Haagerup subgroups. Classify all $G_n$-invariant von Neumann subalgebras in $L(G_n)$ for all $n\geq 2$.
\end{question}

Second, we may consider a general question on classifying all $H$-invariant von Neumann subalgebras in a tracial von Neumann algebra $(M,\tau)$ for a suitable choice of $M$ and a countable subgroup $H\subseteq \mathcal{U}(M)$. In the same spirit of \cite[Conjecture]{ah}, we ask the following questions. We remark that in all these questions intermediate von Neumann subalgebras between $L(G)$ and $M$ are classified by \cites{cd, wit, jiang_jot}.
\begin{question}
Let $G=F_2$ the non-abelian free group on two generators and $\{G_n\}_{n\geq 1}$ be a sequence of decreasing normal subgroups in $G$ with trivial intersection. Consider the associated profinite action $G\curvearrowright X:=\lim\limits_{\leftarrow}G/G_n$. Classify all $G$-invariant von Neumann subalgebras in $M=L^{\infty}(X)\rtimes G$. 
\end{question}
\begin{question}
Let $\Gamma=\mathbb{Z}^2\rtimes_{\rho} G$, where $\rho: G=\mathbb{Z}^2\rtimes SL_2(\mathbb{Z})\rightarrow Aut(\mathbb{Z}^2)$ is the composition of the quotient map  $\mathbb{Z}^2\rtimes SL_2(\mathbb{Z})\twoheadrightarrow SL_2(\mathbb{Z})$ with the inclusion $SL_2(\mathbb{Z})\hookrightarrow Aut(\mathbb{Z}^2)$. Classify all $G$-invariant von Neumann subalgebras in $M=L(\Gamma)$.
\end{question}

\begin{question}
Let $M=L^{\infty}(X,\mu)\rtimes G$, where $G:=PSL_2(\mathbb{Z})\curvearrowright (X,\mu)$ (as studied in \cite{jiang_jot}) denotes the quotient of $SL_2(\mathbb{Z})\curvearrowright \mathbb{T}^2$ by modding out the central subgroup action $\{\pm id\}\curvearrowright \mathbb{T}^2$ and then the kernel of the action. Is every $G$-invariant von Neumann subalgebra in $M$  of the form $L^{\infty}(Y,\nu)\rtimes H$ for some normal subgroup $H\lhd PSL_2(\mathbb{Z})$ and a quotient action $PSL_2(\mathbb{Z})\curvearrowright (Y,\nu)$ of the action $PSL_2(\mathbb{Z})\curvearrowright (X,\mu)$?
\end{question}


\begin{thebibliography}{9}


\bib{agv}{article}{
   author={Ab\'{e}rt, M.},
   author={Glasner, Y.},
   author={Vir\'{a}g, B.},
   title={Kesten's theorem for invariant random subgroups},
   journal={Duke Math. J.},
   volume={163},
   date={2014},
   number={3},
   pages={465--488},}

\bib{ab}{article}{
   author={Alekseev, V.},
   author={Brugger, R.},
   title={A rigidity result for normalized subfactors},
   journal={J. Operator Theory},
   volume={86},
   date={2021},
   number={1},
   pages={3--15},}


\bib{ah}{article}{
author={Amrutam, T.},
author={Hartman, Y.},
title={Subalgebras, subgroups and singularity},
journal={Bull. Lond. Math. Soc.},
   volume={56},
   date={2024},
   number={1},
   pages={380--395},
}

\bib{aho}{article}{
author={Amrutam, T.},
author={Hartman, Y.},
author={Oppelmayer, H.},
title={On the amenable subalgebras of group von Neumann algebras},
 journal={J. Funct. Anal.},
   volume={288},
   date={2025},
   number={2},
   pages={Paper No. 110718, 20 pp},
}
\bib{aj}{article}{
   author={Amrutam, T.},
   author={Jiang, Y.},
   title={On invariant von Neumann subalgebras rigidity property},
   journal={J. Funct. Anal.},
   volume={284},
   date={2023},
   number={5},
   pages={Paper No. 109804},}


\bib{ap}{book}{
author={Anantharaman-Delaroche, C.},
author={Popa, S.},
title={An introduction to II$_1$ factors},
status={available at \url{https://www.math.ucla.edu/~popa/Books/IIun.pdf}},
year={2019},
}

\bib{bd}{book}{
   author={Bekka, B.},
   author={de la Harpe, P.},
   title={Unitary representations of groups, duals, and characters},
   series={Mathematical Surveys and Monographs},
   volume={250},
   publisher={American Mathematical Society, Providence, RI},
   date={2020},
   pages={xi+474},}

\bib{bo}{book}{
   author={Brown, N. P.},
   author={Ozawa, N.},
   title={$C^*$-algebras and finite-dimensional approximations},
   series={Graduate Studies in Mathematics},
   volume={88},
   publisher={American Mathematical Society, Providence, RI},
   date={2008},
   pages={xvi+509},}

\bib{cd}{article}{
   author={Chifan, I.},
   author={Das, S.},
   title={Rigidity results for von Neumann algebras arising from mixing
   extensions of profinite actions of groups on probability spaces},
   journal={Math. Ann.},
   volume={378},
   date={2020},
   number={3-4},
   pages={907--950},}

\bib{cds}{article}{
   author={Chifan, I.},
   author={Das, S.},
   author={Sun, B.},
   title={Invariant subalgebras of von Neumann algebras arising from
   negatively curved groups},
   journal={J. Funct. Anal.},
   volume={285},
   date={2023},
   number={9},
   pages={Paper No. 110098},}


\bib{cs}{article}{
   author={Chifan, I.},
   author={Sinclair, T.},
   title={On the structural theory of ${\rm II}_1$ factors of negatively
   curved groups},
   journal={Ann. Sci. \'{E}c. Norm. Sup\'{e}r. (4)},
   volume={46},
   date={2013},
   number={1},
   pages={1--33},}
   
   
    \bib{gel}{article}{
   author={Gelander, T.},
   title={A view on invariant random subgroups and lattices},
   conference={
      title={Proceedings of the International Congress of
      Mathematicians---Rio de Janeiro 2018. Vol. II. Invited lectures},
   },
   book={
      publisher={World Sci. Publ., Hackensack, NJ},
   },
   date={2018},
   pages={1321--1344},}
   
   \bib{ioana_subequivalence}{article}{
   author={Ioana, A.},
   title={Relative property (T) for the subequivalence relations induced by
   the action of ${\rm SL}_2(\Bbb Z)$ on $\Bbb T^2$},
   journal={Adv. Math.},
   volume={224},
   date={2010},
   number={4},
   pages={1589--1617},}
   
   
   \bib{ioana_gafa}{article}{
   author={Ioana, A.},
   title={Uniqueness of the group measure space decomposition for Popa's
   $\scr{HT}$ factors},
   journal={Geom. Funct. Anal.},
   volume={22},
   date={2012},
   number={3},
   pages={699--732},}
   
   
   \bib{ioana_icm}{article}{
   author={Ioana, A.},
   title={Rigidity for von Neumann algebras},
   conference={
      title={Proceedings of the International Congress of
      Mathematicians---Rio de Janeiro 2018. Vol. III. Invited lectures},
   },
   book={
      publisher={World Sci. Publ., Hackensack, NJ},
   },
   date={2018},
   pages={1639--1672},}
   
\bib{jiang_jot}{article}{
   author={Jiang, Y.},
   title={Maximal Haagerup subalgebras in $L(\Bbb Z^2\rtimes SL_2(\Bbb Z))$},
   journal={J. Operator Theory},
   volume={86},
   date={2021},
   number={1},
   pages={203--230},}   
   
   \bib{jiangskalski}{article}{
   author={Jiang, Y.},
   author={Skalski, A.},
   title={Maximal subgroups and von Neumann subalgebras with the Haagerup
   property},
   journal={Groups Geom. Dyn.},
   volume={15},
   date={2021},
   number={3},
   pages={849--892},}

\bib{jz}{article}{
author={Jiang, Y.},
author={Zhou, X.},
title={An example of an infinite amenable group with the ISR property},
journal={Math. Z.},
   volume={307},
   date={2024},
   number={2},
   pages={Paper No. 23, 11 pp},
}

\bib{jv}{article}{
author={Jolissaint, P.},
author={Valette, A.},
title={The linear $SL_2(\mathbb{Z})$-action on $\mathbb{T}^n$: ergodic and von Neumann algebraic aspects},
year={2023},
status={arXiv: 2311.02683},
}


\bib{kp}{article}{
   author={Kalantar, M.},
   author={Panagopoulos, N.},
   title={On invariant subalgebras of group and von Neumann algebras},
   journal={Ergodic Theory Dynam. Systems},
   volume={43},
   date={2023},
   number={10},
   pages={3341--3353},}




\bib{ozawa_imrn}{article}{
   author={Ozawa, N.},
   title={A Kurosh-type theorem for type $\rm II_1$ factors},
   journal={Int. Math. Res. Not.},
   date={2006},
   pages={Art. ID 97560, 21},}
   
\bib{ozawa_hokkaido}{article}{
   author={Ozawa, N.},
   title={An example of a solid von Neumann algebra},
   journal={Hokkaido Math. J.},
   volume={38},
   date={2009},
   number={3},
   pages={557--561},}

\bib{park}{article}{
   author={Park, K. K.},
   title={${\rm GL}(2,{\bf Z})$ action on a two torus},
   journal={Proc. Amer. Math. Soc.},
   volume={114},
   date={1992},
   number={4},
   pages={955--963},}
   
   
   \bib{popa_betti}{article}{
   author={Popa, S.},
   title={On a class of type II$_1$ factors with Betti numbers
   invariants},
   journal={Ann. of Math. (2)},
   volume={163},
   date={2006},
   number={3},
   pages={809--899},}

\bib{popa_icm}{article}{
   author={Popa, S.},
   title={Deformation and rigidity for group actions and von Neumann
   algebras},
   conference={
      title={International Congress of Mathematicians. Vol. I},
   },
   book={
      publisher={Eur. Math. Soc., Z\"{u}rich},
   },
   date={2007},
   pages={445--477},}

\bib{pv_crelle}{article}{
   author={Popa, S.},
   author={Vaes, S.},
   title={Unique Cartan decomposition for II$_{1}$ factors arising from
   arbitrary actions of hyperbolic groups},
   journal={J. Reine Angew. Math.},
   volume={694},
   date={2014},
   pages={215--239},}

\bib{sako}{article}{
   author={Sako, H.},
   title={Measure equivalence rigidity and bi-exactness of groups},
   journal={J. Funct. Anal.},
   volume={257},
   date={2009},
   number={10},
   pages={3167--3202},}   
   
   \bib{vaes_icm}{article}{
   author={Vaes, S.},
   title={Rigidity for von Neumann algebras and their invariants},
   conference={
      title={Proceedings of the International Congress of Mathematicians.
      Volume III},
   },
   book={
      publisher={Hindustan Book Agency, New Delhi},
   },
   date={2010},
   pages={1624--1650},}

\bib{val}{article}{
author={Valette, A.},
title={Maximal Haagerup subgroups in $\mathbb{Z}^{n+1}\rtimes_{\rho_n}GL_2(\mathbb{Z})$},
journal={Studia Math.},
   volume={275},
   date={2024},
   number={3},
   pages={263--283},
}


\bib{wit}{article}{
   author={Witte, D.},
   title={Measurable quotients of unipotent translations on homogeneous
   spaces},
   journal={Trans. Amer. Math. Soc.},
   volume={345},
   date={1994},
   number={2},
   pages={577--594},}
\end{thebibliography}
\end{document}